\documentclass[11pt]{amsart}

\newtheorem{theorem}{Theorem}[section]

\newtheorem{corollary}[theorem]{Corollary}

\newtheorem{lemma}[theorem]{Lemma}

\newtheorem{proposition}[theorem]{Proposition}

\usepackage[colorlinks, bookmarks=true]{hyperref}
\usepackage{color,graphicx,shortvrb}
\usepackage[active]{srcltx} %SRC Specials for DVI search

\def\11{\textbf{$1$}}

\usepackage{enumerate}
\usepackage{amsmath}
\usepackage{amssymb}
\usepackage[latin 1]{inputenc}

\begin{document}

\title[Weak-2-local $^*$-derivations on $B(H)$]{Quasi-linear functionals determined by weak-2-local $^*$-derivations on $B(H)$}

\author[Niazi]{Mohsen Niazi}
\address{Department of Mathematics, Faculty of Mathematical and Statistical Sciences, University of Birjand, Birjand, Iran.}
\curraddr{Departamento de An{\'a}lisis Matem{\'a}tico, Facultad de
Ciencias, Universidad de Granada, 18071 Granada, Spain.}
\email{niazi@birjand.ac.ir}

\author[Peralta]{Antonio M. Peralta}
\address{Departamento de An{\'a}lisis Matem{\'a}tico, Facultad de
Ciencias, Universidad de Granada, 18071 Granada, Spain.}
\curraddr{Visiting Professor at Department of Mathematics, College of Science, King Saud University, P.O.Box 2455-5, Riyadh-11451, Kingdom of Saudi Arabia.}
\email{aperalta@ugr.es}

\thanks{Authors partially supported by the Spanish Ministry of Economy and Competitiveness project no. MTM2014-58984-P, and Junta de Andaluc\'{\i}a grant FQM375. Second author partially supported by the Deanship of Scientific Research at King Saud University (Saudi Arabia) research group no. RG-1435-020. The first author acknowledges the partial financial support from the IEMath-GR program for visits of young talented researchers.}


%\date{December 2nd, 2014}

\begin{abstract} We prove that, for every separable complex Hilbert space $H$, every weak-2-local $^*$-derivation on $B(H)$ is a linear $^*$-derivation. We also establish that every {\rm(}non-necessarily linear nor continuous{\rm)} weak-2-local derivation on a finite dimensional C$^*$-algebra is a linear derivation.
\end{abstract}


\maketitle
 \thispagestyle{empty}

\section{Introduction}

The Mackey-Gleason theorem and its subsequent generalizations constitute some of the most influencing results in axiomatic theory of quantum mechanics, and led the researchers to develop many interesting applications to mathematics (compare the monograph \cite{Dvurech}). A renewed mathematical interest in the Mackey-Gleason theorem  becomes more evident after recent applications of these results to determine when a 2-local $^*$-homomorphism or a 2-local derivation on a von Neumann algebra is a linear $^*$-homomorphism or a linear derivation, respectively (cf. \cite{BurFerGarPe2014RACSAM, BurFerGarPe2014JMAA, AyuKuday2014} or \cite{AyuKudPe2014}).\smallskip

In \cite{BenAliPeraltaRamirez}, A. Ben Ali Essaleh, M.I. Ram{\'i}rez and the second author of this note introduce a weak variant of Kadison's notion of local derivations. We recall that, according to Kadison's definition, a linear mapping $T$ from a Banach algebra $A$ into a $A$-bimodule $X$ is said to be a \emph{local derivation} if for every $a$ in $A$, there exists a derivation $D_{a}: A\to X$, depending on $a$, such that $T(a)= D_{a} (a)$ (see \cite{Kad90}). If for every $a$ in $A$, and every $\phi\in X^*$ there exists a derivation $D_{a,\phi}: A\to X$, depending on $a$ and $\phi$, such that $\phi T(a)= \phi D_{a,\phi} (a)$, we say that $T$ is a weak-local derivation. It is due to B.E. Johnson that every local derivation from a C$^*$-algebra $A$ into a Banach $A$-bimodule is a derivation (see \cite{John01}). For the wider notion of weak-local derivations, it is proved in \cite[Theorem 3.4]{BenAliPeraltaRamirez} that every weak-local derivation on a C$^*$-algebra is a derivation.\smallskip

In the setting of non-necessarily linear maps, P. \v{S}emrl defined the notion of 2-local derivations in \cite{Semrl97}. Let $A$ be a Banach algebra, a non-necessarily linear mapping $\Delta :A \to A$, is said to be a 2-local derivation if for every $a,b\in A$, there exists a (linear) derivation $D_{a,b}: A\to X$, depending on $a$ and $b$, such that $T(a)= D_{a,b} (a)$ and $T(b)= D_{a,b} (b)$. For an infinite-dimensional separable Hilbert space $H$, \v{S}emrl proves that every 2-local derivation $T : B(H) \to B(H)$ (no linearity or continuity of $T$ is assumed) is a derivation \cite[Theorem 2]{Semrl97}. The most general conclusion in this line, due to S. Ayupov and K. Kudaybergenov, establishes that every 2-local derivation on an arbitrary von Neumann algebra is a derivation (see \cite{AyuKuday2014} and \cite{AyuKudPe2014}).\smallskip

In an attempt to study a weak version of (non-necessarily linear) 2-local derivations on C$^*$-algebras, we recently introduced the following definition. Let $A$ be a C$^*$-algebra. Following \cite[Definition 1.1]{NiPe}, a {\rm(}non-necessarily linear{\rm)} mapping $\Delta: A\to A$ is said to be a weak-2-local derivation {\rm(}respectively, a weak-2-local $^*$-derivation{\rm)} on $A$ if for every $a,b\in A$ and $\phi\in A^*$ there exists a derivation {\rm(}respectively, a $^*$-derivation{\rm)} $D_{a,b,\phi}: A\to A$, depending on $a$, $b$ and $\phi$, such that $\phi \Delta (a) = \phi D_{a,b,\phi} (a)$ and $\phi \Delta (b) = \phi D_{a,b,\phi} (b)$.\smallskip

The above notion of weak-2-local derivations on C$^*$-algebras is probably one of the weakest possible notions and gives, a priori, a very general class of maps. The usual techniques employed in previous papers are useless to deal with weak-2-local derivations on C$^*$-algebras. As in the historical forerunners, we studied first weak-2-local derivations on matrix algebras and finite dimensional C$^*$-algebras, and we prove that every  weak-2-local $^*$-derivation on a finite dimensional C$^*$-algebra is a linear derivation (cf. \cite[Corollary 3.12]{NiPe}). The question whether every weak-2-local derivation on a matrix algebra or on a finite dimensional C$^*$-algebra was left as an open problem.\smallskip

In this note, we resume the study of weak-2-local derivations on matrix algebras. Section \ref{Sec: matrix algebras} is devoted to present a new algebraic approach to prove that every weak-2-local derivation on a matrix algebra is a derivation. This result plays an important role in Section \ref{Sec: weak-2-local derivations on B(H)}, where we develop the first result for weak-2-local derivations on infinite dimensional C$^*$-algebras. The main result of the paper shows that, for a separable complex Hilbert space $H$, every weak-2-local $^*$-derivation on the C$^*$-algebra, $B(H)$, of all bounded linear operators on $H$, is a linear derivation (cf. Theorem \ref{t weak-2-local *derivations on B(l2)}).\smallskip

Our strategy is based on a different approach to the Bunce-Wright-Mackey-Gleason theorem and related studies. The circle of ideas around the Mackey-Gleason theorem includes several studies on quasi-linear functionals on a C$^*$-algebra $A$ (compare \cite{Aarnes70,BuWri96}). For each self-adjoint element $x$ in $A$, the symbol $A_{x}$ will denote the C$^*$-subalgebra of $A$ generated by $x$. Let $X$ be a Banach space. Following the notion introduced by J.F. Aarnes in \cite{Aarnes70}, we shall say that a \emph{quasi-linear operator} from $A$ into $X$ is a function $\mu : A \to  X$ satisfying:\begin{enumerate}[$(a)$] \item $\mu|_{A_x}: A_x \to X$ is a linear mapping for each self adjoint element $x \in A$;
\item $\mu(a+i b) = \mu(a) + i \mu(b),$ when $a$ and $b$ are self adjoint elements in $A$.
\end{enumerate} If in addition
\begin{enumerate}[$(c)$] \item  $\sup \{\|\mu (a)\|: a \in  A, \|a\|\leq 1\}<\infty,$ then we say that $\mu$ is bounded.
\end{enumerate}

One of the results developed in this note shows that given a complex Hilbert space $H$, every weak-2-local $^*$-derivation $\Delta: K(H)\to K(H)$ is a quasi-linear operator on $K(H)$ (cf. Theorem \ref{t derivation on Ka} and Corollary \ref{c quasi-linear functional on K(H)}). This is a key result to prove that, for a {\rm(}non-necessarily linear nor continuous{\rm)} weak-2-local derivation $\Delta: B(H)\to B(H)$, where $H$ is a separable complex Hilbert space, the mapping $\mathcal{P} (B(H))\to B(H),$ $p\mapsto \Delta(p)$ is a completely additive measure on the lattice, $\mathcal{P} (B(H)),$ of all projections in $B(H)$ (cf. the proof of Theorem \ref{t weak-2-local derivations on B(l2)}). Under these conditions, the Bunce-Wright-Mackey-Gleason and the Dorofeev-Shertsnev theorems can be appropriately applied to establish our main result.\smallskip\smallskip

\noindent\textbf{Notation}\smallskip\smallskip

Throughout the paper, given a Banach space $X$, we consider $X$ as a closed subspace of $X^{**},$ via its natural isometric embedding. Given a closed subspace $Y$ of $X$ we shall identify the weak$^*$-closure, $\overline{Y}^{\sigma(X^{**},X^{*})}$, of $Y$ in $X^{**}$, with $Y^{**}$. Throughout this note, we usually write $M_n$ to denote the C$^*$-algebra $M_n(\mathbb{C})$ of all $n\times n$ square matrices with entries in $\mathbb{C}$.

\section{Weak-2-local derivations on finite dimensional C$^*$-algebras}\label{Sec: matrix algebras}

In this section we prove that every weak-2-local derivation on a finite dimensional C$^*$-algebra is a linear derivation, this solves a problem we left open in \cite{NiPe}. The proof is extremely technical and uses a fundamentally algebraic approach. We split the arguments in a series of technical lemmas to facilitate our proof.\smallskip

Suppose $p_1, \ldots, p_n$ are mutually orthogonal minimal projections in $M_n.$ Given $i,j$ in $\{1,\ldots, n\}$, we shall denote by $e_{ij}$ the unique minimal partial isometry in $M_n$ satisfying $e_{ij}^* e_{ij} = p_j$ and $e_{ij} e_{ij}^* = p_i.$ The symbol $\phi_{ij}$ will denote the unique norm-one functional in $M_n^*$ satisfying $\phi_{ij} (e_{ij})=1$. Throughout this section, we shall frequently apply the identity:
\begin{equation}\label{eq [z,eij]}
  [z, e_{ij}] =(z_{ii} - z_{jj}) e_{ij} + \sum_{k=1, k\ne i}^n z_{ki} e_{kj} - \sum_{k=1, k\neq j}^n z_{jk} e_{ik},
\end{equation}
which is valid for every matrix $z=(z_{ij})\in M_n,$ and every $(i,j)\in \{1,\ldots,n\}.$ In the case $p_j=e_{jj},$ the identity \eqref{eq [z,eij]} writes in the form:
\begin{equation}\label{eq [z,pk]}
  [z,p_j] = \sum_{k=1, k\neq j}^{n} z_{kj} e_{kj} - z_{jk} e_{jk}\,.
\end{equation}

\begin{lemma}\label{l D} Let $\Delta: M_n\to M_n$ be a weak-2-local derivation. Suppose $p_1, \ldots, p_n$ are mutually orthogonal minimal projections in $M_n$. Let $q=1 - p_n,$ and let $\mathfrak{D}$ be a subset of $\{1,\ldots,n-1\}\times\{1,\ldots,n-1\},$ which contains the diagonal {\rm(}i.e. $(j,j)\in \mathfrak{D}$, for every $j\in \{1,\ldots,n-1\}${\rm)}. If $$q\Delta\left(\sum_{(i,j)\in \mathfrak{D}} \lambda_{ij}e_{ij}\right)p_n=0,$$ for every $\lambda_{ij}\in\mathbb{C}$, then $$q\Delta\left(\sum_{(i,j)\in \mathfrak{D}\cup\{(i_0,j_0)\}} \lambda_{ij}e_{ij}\right)p_n=0,$$ for every $\lambda_{ij}\in\mathbb{C}$
 and for every $i_0,j_0\in \{1,\ldots, n-1\},$ with $(i_0,j_0)\notin \mathfrak{D}.$
\end{lemma}

\begin{proof} Let us fix $\displaystyle a=\sum_{(i,j)\in \mathfrak{D}} \lambda_{ij}e_{ij},$ where  $\lambda_{ij}\in\mathbb{C}$. We are assuming $q\Delta (a) p_n =0$. It is easy to see from \eqref{eq [z,eij]} that $\phi_{in}[z,e_{i_0j_0}]=0,$ for every $z\in M_n,$ $1\leq i\leq n-1,\ i\ne i_0.$ Thus, we have $$\phi_{in}[z,a + \lambda_{i_0j_0}e_{i_0j_0}] = \phi_{in}[z,a],$$ for every $z\in M_n,$ $1\leq i\leq n-1,\ i\ne i_0.$
Combining this identity with the weak-2-local behavior of $\Delta$ at $\phi_{in}$ and the points $a + \lambda_{i_0j_0}e_{i_0j_0}$ and $a,$ we prove, from the assumptions, that
  \begin{equation}\label{eq 0310 1}
    \phi_{in}\Delta(a + \lambda_{i_0j_0}e_{i_0j_0})=\phi_{in}\Delta(a)=0,\ \ (1\leq i\leq n-1,\ i\ne i_0).
  \end{equation}

Take the functional $\phi = \phi_{i_0n} + \phi_{j_0n}.$ Considering identity \eqref{eq [z,eij]}, for $[z,e_{i_0j_0}],$ and identity \eqref{eq [z,pk]}, for $[z,p_{j_0}],$ it is not hard to see that $\phi[z,e_{i_0j_0}]=\phi[z,p_{j_0}]$, for every $z\in M_n.$ Therefore, $$\phi[z,a + \lambda_{i_0j_0}e_{i_0j_0}] = \phi[z,a + \lambda_{i_0j_0}p_{j_0}].$$ This identity combined with the weak-2-local property of $\Delta$ at $\phi$ and the points $a + \lambda_{i_0j_0}e_{i_0j_0}$ and $a + \lambda_{i_0 j_0}p_{j_0},$ show that
$$\phi\Delta(a+\lambda_{i_0j_0}e_{i_0j_0})=\phi\Delta(a + \lambda_{i_0 j_0}p_{j_0}).$$ Since $\mathfrak{D}$ contains the diagonal of $\{1,\ldots,n-1\}\times\{1,\ldots,n-1\},$ the element $a + \lambda_{i_0 j_0}p_{j_0}$ writes in the form $\displaystyle a + \lambda_{i_0 j_0}p_{j_0} =\sum_{(i,j)\in \mathfrak{D}} \mu_{ij} e_{ij},$ for suitable $\mu_{i,j}\in \mathbb{C}.$ It follows from the hypothesis that $q \Delta(a + \lambda_{i_0 j_0}p_{j_0}) p_n=0$, and hence $0=\phi\Delta(a + \lambda_{i_0 j_0}p_{j_0})=\phi\Delta(a+\lambda_{i_0j_0}e_{i_0j_0}).$   Since $j_0\ne i_0,$ we deduce from \eqref{eq 0310 1} that $\phi_{j_0n}\Delta(a+\lambda_{i_0j_0}e_{i_0j_0})=0,$ and hence $$\phi_{i_0n}\Delta(a+\lambda_{i_0j_0}e_{i_0j_0})=0,$$ which combined with \eqref{eq 0310 1}, completes the proof.
\end{proof}

Let $A$ and $B$ be C$^*$-algebras. Given a mapping $\Delta: A\to B$, we define a new mapping $\Delta^{\sharp}: A\to B$ determined by the expression $\Delta^{\sharp} (a) = \Delta(a^*)^*$ ($a\in A$). The mapping $\Delta$ is called symmetric if $\Delta^{\sharp} =\Delta.$ It is clear that $\Delta$ is linear if and only if $\Delta^{\sharp}$ is.

\begin{lemma}\label{l q Delta(qaq)pn} Let $\Delta: M_n\to M_n$ be a symmetric {\rm(}i.e. $\Delta^\sharp=\Delta${\rm)} weak-2-local derivation. Suppose $p_1, \ldots, p_n$ are mutually orthogonal minimal projections in $M_n$, and let $q=1-p_n$. If $\Delta (p_j) = 0$, for every $j=1,\ldots,n,$ then
  $$q\Delta(qaq)p_n=0=p_n\Delta(qaq)q,$$ for every $a\in M_n.$
\end{lemma}

\begin{proof} Let $\mathfrak{D}_0$ be the diagonal subset of $\{1,\ldots,n-1\}\times\{1,\ldots,n-1\}.$ We observe that $\{e_{ij} : (i,j)\in \mathfrak{D}_0 \}$ is a set of mutually orthogonal projections in $M_n$. Thus, Proposition 3.4 in \cite{NiPe} and the hypothesis imply that $$q\Delta\left(\sum_{(i,j)\in \mathfrak{D}_0} \lambda_{ij}e_{ij}\right)p_n= q\left(\sum_{(i,j)\in \mathfrak{D}_0} \lambda_{ij} \Delta (e_{ij}) \right)p_n=0,$$ for every $\lambda_{ij}\in\mathbb{C}$. Now, applying Lemma \ref{l D} a finite number of times we deduce that $q\Delta(qaq)p_n=0,$ for every $a\in M_n.$\smallskip

For the second identity we observe that, since $\Delta$ is symmetric, we have $$p_n\Delta(qaq)q = \left(q \Delta(qaq)^* p_n \right)^* = \left(q \Delta(qa^*q) p_n \right)^* = 0,$$ for every $a\in M_n.$
\end{proof}

Lemma 2.5 in \cite{NiPe} proves that for every weak-2-local derivation $\Delta$ on a unital C$^*$-algebra $A$, we have $\Delta (1-a)= -\Delta (a)$, for every $a\in A$. Weak-2-local derivations are 1-homogeneous so, given a projection $p\in A$ and $\lambda\in \mathbb{C}\backslash\{0\}$, we have $$\Delta( a + \lambda p) = \lambda \Delta( \lambda^{-1} a + p) =\lambda \Delta( \lambda^{-1} a + p -1) = \Delta(a-\lambda (1-p)).$$ Since the above equality is obviously true for $\lambda=0$, we have:

\begin{lemma}\label{l Delta(a+lambda1)} Let $\Delta: A\to A$ be a weak-2-local derivation on a unital C$^*$-algebra. Then for each projection $p\in A$, $a\in A,$ and $\lambda\in \mathbb{C},$ we have $$\Delta(a+\lambda p)=\Delta(a-\lambda (1-p)).$$ $\hfill\Box$
\end{lemma}

\begin{lemma}\label{l Delta(a+ekn)} Let $\Delta: M_n\to M_n$ be a weak-2-local derivation, let $p_1, \ldots, p_n$ be mutually orthogonal minimal projections in $M_n$, and let $q=1-p_n.$ Suppose $R$ is a subset of $\{1,\ldots,n-1\}.$ We set $\displaystyle r=\sum_{i\in R}p_i$ when $R\neq \emptyset$, and $r=0$ when  $R=\emptyset$. Let us assume that $\Delta (qaq + rap_n) = 0,$ for every $a\in M_n.$ Then $$\Delta (qaq + rap_n + \lambda e_{kn}) = p_k\Delta (qaq + rap_n + \lambda e_{kn})p_n,$$ for every $a\in M_n,$ $\lambda\in\mathbb{C},$ and $1\leq k\leq n-1$.
\end{lemma}

\begin{proof} Let us pick $a\in M_n,$ and $\lambda\in\mathbb{C}.$ Fix an arbitrary $\phi\in M_n^*$ satisfying $\phi = (1-p_k)\phi (1-p_n)$.  Since, for each $z$ in $M_n$, we have $$(1-p_k)[z,e_{kn}](1-p_n)=0,$$ the weak-2-local property of $\Delta$ at $\phi$, $qaq + rap_n + \lambda e_{kn}$ and $qaq + rap_n,$ shows that $$\phi \Delta(qaq + rap_n + \lambda e_{kn}) = \phi \Delta(qaq + rap_n)  =0,$$ for every $\phi$ in the above conditions. Therefore,
\begin{equation}\label{eq 0309 2}
(1-p_k)\Delta(qaq + rap_n + \lambda e_{kn})(1-p_n) = 0.
\end{equation}

Take the functional $\phi=\phi_{kj}+\phi_{nj}$ $(1\leq j\leq n-1).$ Since $\phi[z,e_{kn}]=\phi[z,p_n]$ (compare \eqref{eq [z,eij]} and \eqref{eq [z,pk]}), by adding appropriate elements in both sides of this equality and using the weak-2-local behavior of $\Delta$ at the points $qaq + rap_n + \lambda e_{kn}$ and $qaq + rap_n + \lambda p_n$, we obtain
$$\phi\Delta(qaq + rap_n + \lambda e_{kn})=\phi\Delta(qaq + rap_n + \lambda p_n )=\hbox{(by Lemma \ref{l Delta(a+lambda1)})}$$ $$=\phi\Delta(qaq + rap_n - \lambda q)=\phi\Delta(q(a-\lambda)q+r(a-\lambda)p_n)=\hbox{(by hypothesis)} =0$$ Now, identity \eqref{eq 0309 2} implies that $\phi_{nj} \Delta(qaq + rap_n + \lambda e_{kn})=0$, and consequently $\phi_{kj} \Delta(qaq + rap_n + \lambda e_{kn})=0,$ for every $1\leq j\leq n-1.$ Therefore,
\begin{equation}\label{eq 0309 4}
  p_k\Delta(qaq + rap_n + \lambda e_{kn})(1-p_n) = 0.
\end{equation}

Finally, we consider the functional $\phi=\phi_{ik}+\phi_{in}$ $(1\leq i\leq n,i\ne k).$ Since $\phi[z,e_{kn}]=\phi[z,p_k],$ for every $z\in M_n$, the weak-2-local behavior of $\Delta$, at the points $qaq + rap_n + \lambda e_{kn}$ and $qaq + rap_n + \lambda p_k,$ proves that $$\phi\Delta(qaq + rap_n + \lambda e_{kn}) = \phi\Delta(qaq + rap_n + \lambda p_{k})$$ $$= \phi\Delta(q(a+\lambda p_k)q+r(a+\lambda p_k)p_n)
=\hbox{(by hypothesis)}= 0.$$ Applying \eqref{eq 0309 2} we deduce that $\phi_{ik}\Delta(qaq + rap_n + \lambda e_{kn}) = 0$, for every $i\neq k$, and hence $\phi_{in}\Delta(qaq + rap_n + \lambda e_{kn}) = 0,$ equivalently, $$p_i\Delta(qaq + rap_n + \lambda e_{kn}) p_n= 0,$$ for every $1\leq i\leq n, i\ne k.$ The desired statement follows from this final equality combined with \eqref{eq 0309 2} and \eqref{eq 0309 4}. \end{proof}

\begin{lemma}\label{l phi[z,.]} Let $p_1, \ldots, p_n$ be mutually orthogonal minimal projections in $M_n$, and let $q=1-p_n.$ Let $R$ be a subset of $\{1,\ldots,n\}.$ We keep the notation of Lemma \ref{l Delta(a+ekn)}. Then the following statements hold:
\begin{enumerate}[$(i)$]
    \item Let $\phi = \phi_{ln} + \phi_{kn},$ where $1\leq l,k\leq n, l\ne k,$ then $$\phi[z,p_kaq]=\phi[z,e_{lk}aq],$$ for every $z$ and $a$ in $M_n$.
    \item Let $\phi = \phi_{k1} + \phi_{kn},$ where $k\notin R,$ then $$\phi[z,rap_n]=\phi[z,rae_{n1}],$$ for every $z$ and $a$ in $M_n.$
  \end{enumerate}
\end{lemma}

\begin{proof} $(i)$ It is not hard to check that $$\phi_{ln}(zp_kaq)= \phi_{ln}(p_kaqz)= \phi_{kn}(ze_{lk}aq)= \phi_{kn}(e_{lk}aqz)= \phi_{kn}(zp_kaq)=0,$$ $$ \phi_{ln}(ze_{lk}aq)=0,\  \phi_{kn}(p_kaqz)=\sum_{j=1}^{n-1}a_{kj}z_{jn},\hbox{ and } \phi_{ln}(e_{lk}aqz)=\sum_{j=1}^{n-1}a_{kj}z_{jn},$$ for every $a,z\in M_n.$ Thus, $$\phi_{ln}[z,p_kaq]= \phi_{kn}[z,e_{lk}aq]=0, \hbox{ and } \phi_{kn}[z,p_kaq] = \phi_{ln}[z,e_{lk}aq],$$ which proves $(i)$.\smallskip

$(ii)$ If $R=\emptyset$ (i.e. $r=0$) we have nothing to prove, otherwise we have
$$\phi_{k1}(zrap_n)= \phi_{k1}(rap_nz)= \phi_{kn}(zrae_{n1})= \phi_{kn}(rae_{n1}z)=\phi_{kn}(rap_nz)=0,$$
$$ \phi_{k1}(rae_{n1}z)=0,\ \phi_{kn}(zrap_n)=\sum_{j\in R}z_{kj}a_{jn}, \hbox{ and } \phi_{k1}(zrae_{n1})=\sum_{j\in R}z_{kj}a_{jn},$$  for every $a,z\in M_n,$ identities which prove the second statement.
\end{proof}

\begin{lemma}\label{l Delta(a+ekn)=0}
Let $\Delta: M_n\to M_n$ be a weak-2-local derivation, let $p_1, \ldots, p_n$ be mutually orthogonal minimal projections in $M_n$, and let $q=1-p_n.$ 
Suppose $R$ is a subset of $\{1,\ldots,n-1\}.$ We set $\displaystyle r=\sum_{i\in R}p_i$ when $R\neq \emptyset$, and $r=0$ when  $R=\emptyset$.
Let us assume that $\Delta (qaq + rap_n) = 0$, for every $a\in M_n,$ and $\Delta(e_{kn})=0,$ for some $1\leq k\leq n-1.$ Then $$\Delta (qaq + rap_n + \lambda e_{kn}) = 0,$$ for every $a\in M_n,$ and $\lambda\in\mathbb{C}.$
\end{lemma}

\begin{proof} It is easy to see that $\phi_{kn}[z,(1-p_k)qaq]=0,$ for every $z,a\in M_n.$ Thus,
  $$\phi_{kn}[z,(1-p_k)qaq + \lambda e_{kn}] = \phi_{kn}[z,\lambda e_{kn}].$$
  Combining this identity with the weak-2-local property of $\Delta$ we obtain
  \begin{equation}\label{eq 0315 2}
    \phi_{kn}\Delta((1-p_k)qaq + \lambda e_{kn}) = \phi_{kn}\Delta(\lambda e_{kn})=0.
  \end{equation}

By hypothesis, we have $\Delta (qaq + rap_n) = 0$ for every $a\in M_n$, thus, Lemma \ref{l Delta(a+ekn)} implies that
$$\Delta(qaq+ rap_n + \lambda e_{kn})=p_k\Delta(qaq+ rap_n + \lambda e_{kn})p_n,$$ which, in particular, assures that
$$\Delta((1-p_k)qaq + \lambda e_{kn})=p_k\Delta((1-p_k)qaq + \lambda e_{kn})p_n,$$ (we just replace $a$ with $(1-p_k)aq$).
Combining this equality with identity \eqref{eq 0315 2}, we prove
  \begin{equation}\label{eq 0309 5}
    \Delta((1-p_k)qaq + \lambda e_{kn})=0,
  \end{equation}
  for every $a\in M_n.$ \smallskip

Pick $1\leq l \leq n-1, l\ne k,$ and take the functional $\phi = \phi_{ln} + \phi_{kn}.$ By Lemma \ref{l phi[z,.]}$(i)$, we have $\phi[z,p_kqaq] = \phi[z,e_{lk}qaq],$ for every $z,a\in M_n.$ Adding appropriate elements in both sides of this equality and using the weak-2-local behavior of $\Delta$, we obtain
  \begin{equation}\label{eq 0315 3}
    \phi\Delta(qaq + \lambda e_{kn}) = \phi\Delta((1-p_k)qaq + e_{lk}qaq + \lambda e_{kn})
  \end{equation}
  $$=\phi\Delta((1-p_k)q(a + e_{lk}qa)q + \lambda e_{kn})=\hbox{(by \eqref{eq 0309 5})}=0.$$ We can apply Lemma \ref{l Delta(a+ekn)} (with $r=0$), to show that
  \begin{equation}\label{eq 0315 4}
    \Delta(qaq + \lambda e_{kn})=p_k\Delta(qaq + \lambda e_{kn})p_n,
  \end{equation} and hence $\phi_{ln}\Delta(qaq + \lambda e_{kn}) = 0,$ for every $a\in M_n$, which in combination with \eqref{eq 0315 3}, shows that $\phi_{kn}\Delta(qaq + \lambda e_{kn}) = 0.$  Applying \eqref{eq 0315 4}, we deduce that
  \begin{equation}\label{eq 0309 6}
    \Delta(qaq + \lambda e_{kn})=0,\ \ (a\in M_n).
  \end{equation}

Take the functional $\phi = \phi_{k1} + \phi_{kn}.$ Lemma \ref{l phi[z,.]}$(ii)$ shows that $\phi[z,rap_n]=\phi[z,rae_{n1}],$ for every $z,a\in M_n.$ Adding appropriate elements in both sides of this equality and using the weak-2-local behavior of $\Delta$, we obtain
  $$\phi\Delta(qaq + rap_n + \lambda e_{kn}) = \phi\Delta(qaq + rae_{n1} + \lambda e_{kn})$$
  $$=\phi\Delta(q(a + rae_{n1})q + \lambda e_{kn})=\hbox{(by \eqref{eq 0309 6})}=0.$$
Lemma \ref{l Delta(a+ekn)} and the previous identity prove the desired statement.
\end{proof}

\begin{lemma}\label{l Delta(qaq + qapn + pnaq)=0}
Let $\Delta: M_n\to M_n$ be a weak-2-local derivation, let $p_1, \ldots, p_n$ be mutually orthogonal minimal projections in $M_n$, and let $q=1-p_n.$
Suppose that $\Delta^\sharp=\Delta$ and $\Delta (qa)=0,$ for every $a\in M_n.$ Then $$\Delta(qa + p_naq)=0,\ \ (a\in M_n).$$
\end{lemma}

\begin{proof}
Let $a$ be an element in $M_n.$ Let us fix an arbitrary $t\in M_n$, and take the functional $\phi(\cdot)={\rm tr\,}(p_ntq\,\cdot)$ in $M_n^*,$ where tr is the unique unital trace on $M_n$.\smallskip

Since $\phi[z,p_naq]=0,$ for every $z\in M_n,$ it follows that
$$\phi[z,qa + p_naq]=\phi[z,qa],$$ for every $z\in M_n$. Having in mind this identity, we deduce, by the weak-2-local property of $\Delta$ at the points $qa + p_naq$ and  $qa,$ that
  $$\phi\Delta(qa + p_naq)=\phi\Delta(qa)=0.$$ The precise form of $\phi$ implies that $${\rm tr\,}( tq\Delta(qa + p_naq)p_n)= {\rm tr\,}(p_n tq\Delta(qa + p_naq))=\phi\Delta(qa + p_naq) =0,$$ for every $t\in M_n,$ which shows that,
  \begin{equation}\label{eq 0314 1}
    q\Delta(qa + p_naq)p_n=0,\ \ (a\in M_n).
  \end{equation} The condition $\Delta^{\sharp}= \Delta$ implies that $$p_n\Delta(aq + qap_n)q=0, \ \ (a\in M_n).$$ Now, we apply the identity
$aq + qap_n = qaq + p_n aq + qap_n = qa + p_n aq$ to obtain 
\begin{equation}\label{eq 0314 2}
p_n\Delta(q a+ p_n a q)q=0,\ \ (a\in M_n).
\end{equation}

Again, let us fix an arbitrary $t\in M_n$, and take the functional $\phi(\cdot)={\rm tr\,}((qtq + qap_n + p_naq)\,\cdot).$ Since ${\rm tr\,}(x[z,x])=0,$ for every $x,z\in M_n,$ we see that $\phi[z,qtq + qap_n + p_naq]=0,$ and hence $$\phi[z,qa + p_naq] = \phi[z,q(a-t)q] \ (z\in M_n).$$
The weak-2-local property of $\Delta$, at the points $qa + p_naq$ and $q(a-t)q,$ and the above identity, show that \begin{equation}\label{eq 0419 1} \phi\Delta(qa + p_naq) = \phi\Delta(q(a-t)q) = 0.
\end{equation}

Since $\phi(p_n\Delta(qa + p_naq)p_n)=0,$ we deduce from \eqref{eq 0314 1}, \eqref{eq 0314 2}, and \eqref{eq 0419 1} that  $$\phi(q\Delta(qa + p_naq)q) = 0,$$ thus ${\rm tr\,}(tq\Delta(qa + p_naq)q)=0,$ for every $t\in M_n,$ which shows that, $$q\Delta(qa +
  p_naq)q=0.$$ The conclusion of the lemma follows from this identity, together with \eqref{eq 0314 1}, \eqref{eq 0314 2}, and \cite[Lemma 3.1]{NiPe}.
\end{proof}

\begin{lemma}\label{l Delta(sum e1j+ej1)}
Let $\Delta: M_n\to M_n$ be a weak-2-local derivation, let $p_1, \ldots, p_n$ be mutually orthogonal minimal projections in $M_n$ satisfying $\Delta (p_k) = 0$ $(1\leq k\leq n).$ If $\Delta(p_1+e_{1j_0}-e_{i_01}) = 0,$ for some $2\leq i_0,j_0\leq n,$ then $\Delta(e_{i_0j_0})=0,$ where $e_{ij}$ is the unique minimal partial isometry in $M_n$ satisfying $e_{ij}^* e_{ij} = p_j$ and $e_{ij} e_{ij}^* = p_i.$
\end{lemma}

\begin{proof} If $i_0=j_0,$ then the conclusion is clear. So, we fix $i_0\ne j_0$ in $\{2,\ldots,n\}.$ Since $\phi_{ij}[z,e_{i_0j_0}]=0$ for every $z\in M_n,$ and $1\le i,j\leq n,$ $i\ne i_0,$ $j\ne j_0,$ we obtain, from the weak-2-local property of $\Delta,$ that
  \begin{equation}\label{eq 0314 4}
    \phi_{ij}\Delta(e_{i_0j_0})=0,\ \ (1\le i,j\leq n,\ i\ne i_0,\ j\ne j_0).
  \end{equation}

Take the functional $\phi=\phi_{i_0j}+\phi_{j_0j}$ $(1\leq j\leq n, j\ne j_0),$ and consider the identity $\phi[z,e_{i_0j_0}]=\phi[z,p_{j_0}].$ The weak-2-local behavior of $\Delta$ at the points $e_{i_0j_0}$ and $p_{j_0},$ assures that
  $$\phi\Delta(e_{i_0j_0})=\phi\Delta(p_{j_0})=0.$$
  Since $\phi_{j_0j}\Delta(e_{i_0j_0})=0$ (by \eqref{eq 0314 4}), we obtain
  \begin{equation}\label{eq 0314 5}
    \phi_{i_0j}\Delta(e_{i_0j_0})=0,\ \ (1\leq j\leq n, j\ne j_0).
  \end{equation}

Similarly, by taking the functional $\phi=\phi_{ij_0}+\phi_{ii_0}$ $(1\leq i\leq n, i\ne i_0),$ and applying the weak-2-local property of $\Delta$ at the points $e_{i_0j_0}$ and $p_{i_0},$ we get
  \begin{equation}\label{eq 0314 6}
    \phi_{ij_0}\Delta(e_{i_0j_0})=0,\ \ (1\leq i\leq n, i\ne i_0).
  \end{equation}

Take the functional $\phi = \phi_{1j_0} - \phi_{11} + \phi_{i_0j_0} - \phi_{i_01}.$ Using the bilinearity of the Lie product and identities \eqref{eq [z,eij]} and \eqref{eq [z,pk]}, it is not hard to see that $\phi[z,e_{i_0j_0}]=\phi[z,p_1+e_{1j_0}-e_{i_01}],$ for every $z\in M_n.$ Therefore, the weak-2-local property of $\Delta,$ at the points $e_{i_0j_0}$ and $p_1+e_{1j_0}-e_{i_01},$ gives us the following:
  $$\phi\Delta(e_{i_0j_0}) = \phi\Delta(p_1+e_{1j_0}-e_{i_01})=0.$$
The desired statement follows from \eqref{eq 0314 4}, \eqref{eq 0314 5} and \eqref{eq 0314 6}.
\end{proof}

\begin{lemma}\label{l Delta=0}
Let $\Delta: M_n\to M_n$ be a symmetric {\rm(}i.e. $\Delta^\sharp=\Delta${\rm)} weak-2-local derivation, let $p_1, \ldots, p_n$ be mutually orthogonal minimal projections in $M_n$ with $q= 1-p_n$.  Suppose $\Delta(qaq)=0,$ for every $a\in M_n,$ and $\Delta(e_{1n})=0,$ where $e_{ij}$ is the unique minimal partial isometry in $M_n$ satisfying $e_{ij}^* e_{ij} = p_j$ and $e_{ij} e_{ij}^* = p_i.$ Then $\Delta\equiv0.$
\end{lemma}

\begin{proof} Applying Lemmas \ref{l Delta(a+ekn)=0} and \ref{l Delta(sum e1j+ej1)} we deduce, after a finite number of steps, that $\Delta(qa)=0,$ for every $a\in M_n.$ Applying Lemma \ref{l Delta(qaq + qapn + pnaq)=0} we prove that \begin{equation}\label{eq 0420} \Delta (q a + p_n a q) =0, \ (a\in M_n).
 \end{equation}

Finally, $$\Delta (a ) =  \Delta (q a + p_n a q + \lambda_a p_n)= \hbox{(by Lemma \ref{l Delta(a+lambda1)})} = \Delta (q a + p_n a q - \lambda_a q)  $$ $$= \Delta (q (a - \lambda_a) + p_n (a - \lambda_a) q)=\hbox{(by \eqref{eq 0420})} = 0,$$
which completes the proof.
\end{proof}

The next result is a strengthened version of \cite[Proposition 3.8]{NiPe}.

\begin{proposition}\label{p Delta(pk,e1j,ej1)}
Let $\Delta: M_n\to M_n$ be a weak-2-local derivation, let $p_1, \ldots, p_n$ be mutually orthogonal minimal projections in $M_n$, and let $e_{ij}$ denote the unique minimal partial isometry in $M_n$ satisfying $e_{ij}^* e_{ij} = p_j$ and $e_{ij} e_{ij}^* = p_i.$
Then there exists an element $w_0\in M_n,$ such that $$\Delta\left(\sum_{k=1}^n\lambda_kp_k\right) = \left[w_0,\sum_{k=1}^n\lambda_kp_k\right],$$
  for every $\lambda_k\in\mathbb{C},$ and $\Delta(e_{1j})=[w_0,e_{1j}]$ $(2\leq j\leq n).$
\end{proposition}

\begin{proof}
Let $k\in\{1,\ldots,n\}$. Since  $\Delta$ is a weak-2-local derivation, \eqref{eq [z,pk]} shows that
\begin{equation}\label{eq Delta pk}
  \Delta(p_k) = \sum_{i=1, i\neq k}^{n} \alpha_{ik}^{(k)} e_{ik} + \alpha_{ki}^{(k)} e_{ki}\,,
\end{equation}
for suitable $\alpha_{ik}^{(k)},\alpha_{ki}^{(k)}\in\mathbb{C}$. Another application of the weak-2-local behavior of $\Delta,$ at the functional $\phi_{i_0j_0},$ and the points $p_{i_0}$ and $p_{j_0}$, proves that
\begin{equation}\label{eq alpha(i)ij=-alpha(j)ij}
  \alpha_{i_0j_0}^{(i_0)} = -\alpha_{i_0j_0}^{(j_0)}\,,\quad (1\leq i_0,j_0 \leq n, i_0\ne j_0).
\end{equation}
Let us define $$z_0 := - \sum_{i<j}\alpha_{ij}^{(i)}e_{ij} + \sum_{i>j}\alpha_{ij}^{(j)}e_{ij}\,.$$

By \eqref{eq alpha(i)ij=-alpha(j)ij} it is clear that $\Delta (p_k) = [z_0 ,p_k]$, for every $k=1,\ldots,n$. The mapping $\widehat{\Delta} = \Delta -[z_0,.]$ is a weak-2-local derivation, satisfying $$\widehat{\Delta}\left(\sum_{k=1}^n\lambda_kp_k\right)=0,$$ for every $\lambda_1,\ldots,\lambda_n\in \mathbb{C}$ (cf. \cite[Lemma 2.1(a), Proposition 3.4]{NiPe}).\smallskip

Let us fix $2\leq j_0 \leq n$. Combining \eqref{eq [z,eij]}, for $[z, e_{1j_0}]$, with \eqref{eq [z,pk]}, for $[z,p_1],$ and $[z,p_{j_0}],$  and the fact that $\widehat{\Delta}$ is a weak-2-local derivation, we can assert, after an appropriate choosing of functionals $\phi\in M_n^*,$ that there exists $\gamma_{j_0}\in \mathbb{C}$ satisfying $$\widehat{\Delta} (e_{1j_0}) = \gamma_{j_0} e_{1j_0}, \ \ (2\leq j_0 \leq n).$$
If we set $$z_1 := -\sum_{j=2}^n\gamma_j p_j,$$ then $\widehat{\Delta} (e_{1j_0})= [z_1,e_{1j_0}],$ for every $2\leq j_0 \leq n,$ and we further know that $\displaystyle \left[z_1,\sum_{k=1}^n\lambda_kp_k \right]=0,$ for every $\lambda_k\in \mathbb{C}.$ The desired statement is obtained by setting $w_0=z_0+z_1.$
\end{proof}

\begin{proposition}\label{p weak-2-local symmetric derivations on Mn are derivations} Every {\rm(}non-necessarily linear{\rm)} symmetric {\rm(}i.e. $\Delta^\sharp=\Delta${\rm)} weak-2-local derivation on $M_n$ is linear and a derivation.
\end{proposition}

\begin{proof} We shall proceed by induction on $n$. The statement for $n=1$ is clear, while the case $n=2$ is a direct consequence of \cite[Theorem 3.2]{NiPe}. We may, therefore, assume that $n\geq 3.$ Suppose that the desired conclusion is true for $n-1.$ \smallskip

Let $p_1, \ldots, p_n$ be mutually orthogonal minimal projections in $M_n$, and let $e_{ij}$ denote the unique minimal partial isometry in $M_n$ satisfying $e_{ij}^* e_{ij} = p_j$ and $e_{ij} e_{ij}^* = p_i.$. By Proposition \ref{p Delta(pk,e1j,ej1)}, we can assume that
\begin{equation}\label{eq Delta(pk,e1j,ej1)=0}
  \Delta\left(\sum_{k=1}^n\lambda_kp_k\right) = 0 = \Delta(e_{1j}),
\end{equation}
for every $\lambda_k\in\mathbb{C},$ and $2\leq j\leq n.$ Since $\Delta^\sharp=\Delta,$ we obtain
\begin{equation}\label{eq 0314 3}
  \Delta(e_{j1}) = \Delta(e_{1j})^* = 0,\ \ (1\leq j\leq n).
\end{equation}

Let $q=1-p_n.$ Proposition 2.7 in \cite{NiPe} shows that $$q\Delta q:qM_nq\rightarrow qM_nq$$ is a weak-2-local derivation. Since $qM_nq\equiv M_{n-1},$ we deduce, from the induction hypothesis, that $q\Delta q$ is linear on $qM_nq.$ From \eqref{eq 0314 3} it may be concluded that
\begin{equation}\label{eq 0315 1}
  q\Delta\left(\sum_{j=1}^{n-1} \lambda_{1j}e_{1j}+\lambda_{j1}e_{j1}\right)q=0,
\end{equation}
for every $\lambda_{1j},\lambda_{j1}\in\mathbb{C}.$

Combining \eqref{eq Delta(pk,e1j,ej1)=0} with Lemma \ref{l q Delta(qaq)pn}, together with \eqref{eq 0315 1} and \cite[Lemma 3.1]{NiPe}, we prove that
\begin{equation}\label{eq 0420 b} \Delta\left(\sum_{j=1}^{n-1} \lambda_{1j}e_{1j}+\lambda_{j1}e_{j1}\right)=0,\ \ (\lambda_{1j},\lambda_{j1}\in\mathbb{C}).
\end{equation}

Lemma \ref{l Delta(sum e1j+ej1)} now shows, via \eqref{eq 0420 b} and \eqref{eq Delta(pk,e1j,ej1)=0}, that
$$\Delta(e_{ij})=0,\ \ (1\leq i,j\leq n-1).$$ Therefore, for every $a=(a_{ij})\in M_n,$ we have
$$q\Delta(qaq)q = q\Delta\left(\sum_{i,j=1}^{n-1} a_{ij}e_{ij}\right)q = \sum_{i,j=1}^{n-1} a_{ij}q\Delta(e_{ij})q=0.$$
A new application of Lemma \ref{l q Delta(qaq)pn} and \cite[Lemma 3.1]{NiPe} yields $\Delta(qaq)=0,$ for every $a\in M_n.$ Finally, the identity in \eqref{eq Delta(pk,e1j,ej1)=0} can be applied with Lemma \ref{l Delta=0} to conclude the proof.
\end{proof}

Let us mention an important consequence of the above proposition. It is clear that every weak-2-local $^*$-derivation on $M_n$ is a symmetric weak-2-local derivation. However, it is not clear whether the reciprocal implication is, in general true (cf. \cite[comments before Lemma 2.1]{NiPe}). Proposition \ref{p weak-2-local symmetric derivations on Mn are derivations} proves a stronger result by showing that symmetric weak-2-local derivations on $M_n$ are linear $^*$-derivations.\smallskip

The main result of this section is a direct consequence of the above Proposition \ref{p weak-2-local symmetric derivations on Mn are derivations} and \cite[Lemma 2.1]{NiPe}.

\begin{theorem}\label{t weak-2-local derivations on Mn are derivations} Every {\rm(}non-necessarily linear{\rm)} weak-2-local derivation on $M_n$ is a linear derivation.$\hfill\Box$
\end{theorem}

The same arguments applied in the proof of \cite[Corollary 3.12]{NiPe} remain valid to deduce the next corollary from Theorem \ref{t weak-2-local derivations on Mn are derivations} above.

\begin{corollary}\label{c weak-2-local derivations on finite dimensional C-algebras are derivations} Every {\rm(}non-necessarily linear nor continuous{\rm)} weak-2-local derivation on a finite dimensional C$^*$-algebra is a linear derivation.$\hfill\Box$
\end{corollary}

\section{Weak-2-local derivations on $B(H)$}\label{Sec: weak-2-local derivations on B(H)}

Throughout this section, given a complex Hilbert space $H$, the symbols $B(H)$, $K(H)$ and $\mathcal{F} (H)$ will denote the spaces of all bounded, compact and finite-rank operators on $H$.

\begin{proposition}\label{p linearity on finite-ranks}
Let $\Delta: K(H)\to K(H)$ be a weak-2-local derivation, where $H$ is a complex Hilbert space. Then $$\Delta(a+b)=\Delta(a)+\Delta(b),$$
for every $a,b\in \mathcal{F}(H).$
\end{proposition}

\begin{proof} Since $\Delta(a), \Delta(b),$ and $\Delta(a+b)$ are compact operators, for any $\varepsilon>0,$ there exists a finite-rank projection $p,$ such that \begin{equation}\label{eq 0316 1}
    \|\xi -p\xi\|<\varepsilon,\ \ \|\xi-\xi p\|<\varepsilon,
  \end{equation} for every $\xi\in \{\Delta(a), \Delta(b),\Delta(a+b)\}$.\smallskip

Since $a,b$ are finite-rank operators we can also assume that the above projection $p$ also satisfies
\begin{equation}\label{eq 0317 3}
    a=pap\ \ \hbox{and}\ \ b=pbp. \end{equation}

By \cite[Proposition 2.7]{NiPe}, the restriction $$p\Delta p|_{pK(H)p}: pK(H)p\rightarrow pK(H)p$$
is a weak-2-local derivation. We observe that $p K(H)p$ is finite dimensional, thus Theorem \ref{t weak-2-local derivations on Mn are derivations} shows that $p\Delta p|_{pK(H)p}$ is linear. Therefore
\begin{equation}\label{eq 0322} p\Delta(a + b)p=\hbox{ (by  \eqref{eq 0317 3})} = p\Delta(pap + pbp)p
\end{equation} $$= p\Delta(pap)p + p\Delta(pbp)p=\hbox{ (by  \eqref{eq 0317 3})}= p\Delta(a)p + p\Delta(b)p.$$

The inequalities in \eqref{eq 0316 1} assure that $$\|p\Delta(a)p - \Delta(a) \|, \|p\Delta(b)p - \Delta(b) \|, \|p\Delta(a+b)p - \Delta(a+b) \|< 2 \varepsilon,$$ and hence, by \eqref{eq 0322}, we have $$\|\Delta(a+b)-\Delta(a)-\Delta(b)\|\leq  \|p(\Delta(a+b)-\Delta(a)-\Delta(b))p\|$$
  $$+\|(\Delta(a+b)-\Delta(a)-\Delta(b)) - p(\Delta(a+b)-\Delta(a)-\Delta(b))p\| < 6 \varepsilon.$$
The arbitrariness of $\varepsilon>0$ implies that $\Delta(a+b)=\Delta(a)+\Delta(b).$
\end{proof}

Let $D: A\to A$ be a derivation on a C$^*$-algebra. It is known that, $D$ is a continuous operator (cf. \cite[Theorem]{Sak60} or \cite[Theorem 2]{Ringrose72} or \cite[Corollary, page 27]{Kishi76}). We further know that $D^{**} : A^{**}\to A^{**}$ is a (continuous) derivation (cf. \cite[Lemma 3]{Kad66} or \cite[Remark 2.6]{AyuKudPe2014}). %Sakai's theorem (see \cite{Sak66}) assures that $D^{**}$ is an inner derivation, that is, there exists $w\in A^{**}$ satisfying $D^{**} (a) = [w,a]= wa -aw$, for every $a\in A^{**}$.
Therefore, given a projection $p\in A$ and an element $b\in A$ with $p b = b p=0,$ we have \begin{equation}\label{eq pD(b)p=0} pD(b)p = p D((1-p)b) p  = p D^{**}(1-p) b p + p (1-p) D(b) p = 0.
\end{equation}

\begin{lemma}\label{l pDelta(pap+b)p}
Let $\Delta: A\to A$ be a weak-2-local derivation on a C$^*$-algebra. Suppose $p$ is a projection in $A,$ and $b$ is an element in $A$ satisfying $pb=bp=0.$ Then $$p\Delta(a +b)p=p\Delta(a )p,$$ for every $a\in A$. In particular, the identity $$p\Delta(a )p=p\Delta(a - (1-p) a (1-p) )p,$$ holds for every $a\in A$.
\end{lemma}

\begin{proof} Let $\phi$ be a functional in $A^*$ satisfying $\phi=p\phi p.$ Since, for each derivation $D: A\to A,$ we deduce from \eqref{eq pD(b)p=0} that
$$\phi D(a +b) =\phi(pD(a+b)p)=\phi pD(a) p =\phi D(a),$$ it can be concluded from the weak-2-local property of $\Delta,$ at $\phi$ and the points $a+b$ and $a,$ that $$\phi\Delta (a+b)=\phi\Delta(a),$$ for every $a\in A$ and $\phi$ as above.  Lemma 3.5 in \cite{BenAliPeraltaRamirez} implies that $p\Delta(a +b)p=p\Delta(a )p.$
\end{proof}

In \cite[Proposition 2.7]{NiPe}, we prove that if $D:A\rightarrow A$ a derivation {\rm(}respectively, a $^*$-derivation{\rm)} on a C$^*$-algebra, and $p$ is a projection in $A$, then the operator $pDp|_{pAp}:pAp\rightarrow pAp$, $x\mapsto pD(x) p$ is a derivation {\rm(}respectively, a $^*$-derivation{\rm)} on $pAp$. Furthermore, if $\Delta: A\to A$ is a weak-2-local derivation {\rm(}respectively, a weak-2-local $^*$-derivation{\rm)} on $A$, the mapping $p\Delta p|_{pAp} : pAp \to pAp,$ $x\mapsto p \Delta (x) p$ is a weak-2-local derivation {\rm(}respectively, a weak-2-local $^*$-derivation{\rm)} on $pAp$. The next lemma is a consequence of this fact.

\begin{lemma}\label{NiPe p27 for ideal} Let $A$ be a C$^*$-subalgebra of a C$^*$-algebra $B$, and suppose that $A$ is an (two-sided) ideal of $B$.  Let $\Delta: A\to A$ be a weak-2-local derivation {\rm(}respectively, a weak-2-local $^*$-derivation{\rm)} on $A$. Then for each projection $p\in B,$ the mapping $p\Delta p|_{pAp} : pAp \to pAp,$ $x\mapsto p \Delta (x) p$ is a weak-2-local derivation {\rm(}respectively, a weak-2-local $^*$-derivation{\rm)} on $pAp$.$\hfill\Box$
\end{lemma}

\begin{lemma}\label{l ideals are fixed} Let $A$ be a (closed) C$^*$-subalgebra of a von Neumann algebra $M$, and suppose that $A$ is a two-sided ideal of $M$.  Let $\Delta: M\to M$ be a weak-2-local derivation {\rm(}respectively, a weak-2-local $^*$-derivation{\rm)} on $M$. Then $\Delta (A) \subseteq A$ and $\Delta |_{A} : A \to A$ is a weak-2-local derivation {\rm(}respectively, a weak-2-local $^*$-derivation{\rm)} on $A$.
\end{lemma}

\begin{proof} By Sakai's theorem \cite{Sak66} every derivation on $M$ is inner, i.e. for each derivation $D: M \to M$ there exists $w\in M$ satisfying $D(a) =[w,a]$, for every $a\in M$. Consequently, $D(A) = [w,A]\subseteq A$ and $D|_{A} : A \to A$ is a derivation on $A$.\smallskip

Let $\phi\in A^{\circ}$, where $A^{\circ}:= \{\phi \in M^{*}: \phi |_{A} =0\}$ is the polar of $A$ in $M$. For each element $a$ in $A$, and each $\phi\in A^{\circ}$ there exists $w\in M$ satisfying $\phi \Delta (a) = \phi [w,a] =0,$ because $A$ being an ideal implies that $[w,a]\in A$, and $\phi\in A^{\circ}$. We have shown that $\Delta (A) \subseteq \left(A^{\circ}\right)_{\circ} = A$ by the bipolar theorem.
\end{proof}

Our next results determine the behavior of a weak-2-local derivation on a C$^*$-subalgebra generated by a single hermitian compact operator.

\begin{proposition}\label{p linearity on commutative elements} Let $\Delta: K(H)\to K(H)$ be a weak-2-local derivation, where $H$ is a complex Hilbert space. Let $(p_n)$ be a sequence of mutually orthogonal minimal projections in $K(H)$, and let $\mathcal{B}$ denote the commutative C$^*$-subalgebra of $K(H)$ generated by the $p_n$'s. Then $$\Delta(a+b)=\Delta(a)+\Delta(b),$$ for every $a,b\in \mathcal{B}.$
\end{proposition}

\begin{proof} Let us recall that $a,b\in \mathcal{B}$ implies that $a$ and $b$ write in the form $\displaystyle a= \sum_{n=1}^{\infty} \lambda_n p_n$ and $\displaystyle b= \sum_{n=1}^{\infty} \mu_n p_n$, where $(\lambda_n), (\mu_n)\in c_0$. We deduce from these spectral resolutions and the fact that $\Delta(a+b)$, $\Delta(a)$, and $\Delta(b)$ are in $K(H),$ that for each $\varepsilon>0,$ there exists a finite-rank projection $p$ in $K(H)$ satisfying
  \begin{equation}\label{eq 0316 3old}
    \|\xi -p\xi\|<\varepsilon,\ \ \|\xi-\xi p\|<\varepsilon,
  \end{equation} for every $\xi\in \{\Delta(a), \Delta(b),\Delta(a+b)\}$ and
  $$a=pap+p^\perp ap^\perp\ \ \hbox{and}\ \ b=pbp+p^\perp bp^\perp.$$

By Lemma \ref{l pDelta(pap+b)p}, we have
  \begin{equation}\label{eq 0317 1}
    p\Delta(a)p =p\Delta(pap+p^\perp ap^\perp)p = p\Delta(pap)p,
  \end{equation}
  and similarly
  \begin{equation}\label{eq 0317 2}
    p\Delta(b)p = p\Delta(pbp)p,\ \hbox{and}\ p\Delta(a+b)p = p\Delta(p(a+b)p)p.
  \end{equation}

Since, by \cite[Proposition 2.7]{NiPe}, the mapping $$p\Delta p|_{pK(H)p}: pK(H)p\rightarrow pK(H)p,$$
is a weak-2-local derivation, and $pK(H)p\equiv M_{m}$ for a suitable $m\in \mathbb{N}$, Theorem \ref{t weak-2-local derivations on Mn are derivations} shows that $p\Delta p|_{pK(H)p}$ is linear, and hence $$p(\Delta(p(a+b)p) - \Delta(pap) - \Delta(pbp))p = 0,$$ which by \eqref{eq 0317 1} and \eqref{eq 0317 2} implies that
  \begin{equation}\label{eq 0316 4}
    p(\Delta(a+b) - \Delta(a) - \Delta(b))p = 0.
  \end{equation}

Similar arguments to those given in the proof of Proposition \ref{p linearity on finite-ranks} apply, via \eqref{eq 0316 3old} and \eqref{eq 0316 4}, to establish the desired statement.
\end{proof}

Given a symmetric element $a$ in a C$^*$-algebra $A$, the symbol $A_a$ will denote the abelian C$^*$-subalgebra of $A$ generated by $a.$

\begin{corollary}\label{c Ka} Let $\Delta: K(H)\to K(H)$ be a weak-2-local derivation, where $H$ is a complex Hilbert space, and let $a$ be a self-adjoint element in $K(H).$ Then the restriction $\Delta|_{K(H)_a}:K(H)_a\to K(H)$ is linear.$\hfill\Box$
\end{corollary}

Let us observe that, in the hypothesis of the above corollary, although the mapping $\Delta|_{K(H)_a}:K(H)_a\to K(H)$ is linear, we cannot conclude yet that it is continuous, we simply observe that \cite[Theorem 2.1]{BenAliPeraltaRamirez} cannot be applied.\smallskip

Let $a$ be a self-adjoint element in $K(H)$. In the following result we consider $K(H)$ as a $K(H)_a$-bimodule.

\begin{theorem}\label{t derivation on Ka} Let $H$ be a complex Hilbert space, and let $a$ be a compact self-adjoint operator in the C$^*$-algebra $K(H).$ Suppose that $\Delta: K(H)\to K(H)$ be a weak-2-local derivation. Then the mapping $\Delta|_{K(H)_a}: K(H)_a\to K(H)$ is a continuous linear derivation.

Furthermore, let $\mathcal{B}$ denote the commutative C$^*$-subalgebra of $K(H)$ generated by an at most countable family $(p_n)$ of mutually orthogonal minimal projections in $K(H)$. Then $\Delta|_{\mathcal{B}}: \mathcal{B}\to K(H)$ is a continuous linear derivation.
\end{theorem}

\begin{proof} Corollary \ref{c Ka} implies that $\Delta|_{K(H)_a}$ is linear. We shall show next that $\Delta|_{K(H)_a}$ is a Jordan derivation.\smallskip

Let us take a countable family $(p_n)$ of mutually orthogonal minimal projections in $K(H)$ satisfying that $\displaystyle a= \sum_{n=1}^{\infty} \lambda_n p_n,$  where $(\lambda_n)\in c_0$ and $\lambda_n\in \mathbb{R}$ for every $n$. Every element $b$ in $K(H)_a$ writes in the form $\displaystyle b= \sum_{n=1}^{\infty} \mu_n p_n,$ where $(\mu_n)\in c_0.$ Since $\Delta(b)$, $\Delta(b^2)$, $b\Delta(b)$ and $\Delta(b)b$ are compact operators, we deduce, from the spectral resolution of $b$, that for each $\varepsilon>0,$ there exists a finite-rank projection $p$ in $K(H)$ satisfying
  \begin{equation}\label{eq 0316 3}
    \|\xi -p\xi p\|<\varepsilon,
  \end{equation} for every $\xi\in \{ \Delta(b^2), b \Delta(b), \Delta(b) b\}$ and
  \begin{equation}\label{eq 0326 1}
    b=pbp+p^\perp b p^\perp\ \ \hbox{and}\ \ b^2=pb^2 p+p^\perp b^2 p^\perp.
  \end{equation}

Lemma \ref{l pDelta(pap+b)p} combined with \eqref{eq 0326 1} imply that $$p \Delta ( b ) p = p \Delta ( p b p) p, \hbox{ and } p \Delta ( b^2 ) p = p \Delta ( p b^2 p) p.$$

Proposition 2.7 in \cite{NiPe} assures that the mapping $$p\Delta p|_{pK(H)p}: pK(H)p\rightarrow pK(H)p,$$
is a weak-2-local derivation, and since $pK(H)p\equiv M_{m}$ for a suitable $m\in \mathbb{N}$, Theorem \ref{t weak-2-local derivations on Mn are derivations} shows that $p\Delta p|_{pK(H)p}$ is a linear derivation, therefore \begin{equation}
\label{eq 2503 2} p \Delta ((p  b p) (pbp )) p  = p \Delta (p b p) p (pbp) + (pbp) p \Delta (p b p) p.
\end{equation}

Now, it follows from \eqref{eq 0326 1} and Lemma \ref{l pDelta(pap+b)p} that $$p \left(\Delta ( b^2 ) - \Delta ( b ) b - b  \Delta ( b ) \right) p  = p \Delta (p b^2 p) p - p\Delta ( b ) p (p bp)  - (pbp) p  \Delta ( b ) p$$ $$= p \Delta ((p  b p) (pbp )) p  - p \Delta (p b p) p (pbp) - (pbp) p \Delta (p b p) p = \hbox{ (by \eqref{eq 2503 2}) }=0.$$

It can be now deduced from \eqref{eq 0316 3} and \eqref{eq 0326 1} that $$\left\| \Delta (b^2) - \Delta(b) b - b\Delta (b)  \right\| \leq \left\| p (\Delta (b^2) - \Delta(b) b - b\Delta (b) )p  \right\|$$ $$+ \left\| (\Delta (b^2) - \Delta(b) b - b\Delta (b))- p (\Delta (b^2) - \Delta(b) b - b\Delta (b)) p  \right\|<3 \varepsilon.$$ The arbitrariness of $\varepsilon>0 $ proves that $$\Delta (b^2) - \Delta(b) b - b\Delta (b) = 0,$$ for every $b\in K(H)_a.$ Therefore, the mapping $\Delta|_{K(H)_a}: K(H)_a\to K(H)$ is a linear Jordan derivation. Corollary 17 in \cite{PeRu} shows that $\Delta|_{K(H)_a}$ is continuous and \cite[Theorem 6.2]{John96} gives the desired statement. The proof of the second conclusion follows similarly.
\end{proof}

Combining Theorem \ref{t derivation on Ka} with \cite[Lemma 2.3]{NiPe} we can deduce the following result.

\begin{corollary}\label{c quasi-linear functional on K(H)} Let $\Delta: K(H)\to K(H)$ be a weak-2-local $^*$-derivation, where $H$ is a complex Hilbert space. Then $\Delta$ is a quasi-linear operator on $K(H)$.$\hfill\Box$
\end{corollary}

In most of the papers studying 2-local and weak-2-local derivations on von Neumann algebras the arguments rely on ingenious appropriate applications of the Bunce-Wright-Mackey-Gleason theorem \cite{BuWri92} (compare, for example, \cite{AyuKuday2014,KOPR2014} and \cite{AyuKudPe2014}). The just quoted theorem provides the following powerful tool: Let $\mathcal{P} (M)$ denote the lattice of projections in a von
Neumann algebra $M.$
%The following result is implicitly applied in
%\cite{AyuKuday2014}.
Let $X$ be a Banach space. A mapping $\mu:
\mathcal{P} (M)\to X$ is said to be \emph{finitely additive} when
$$
\mu \left(\sum\limits_{i=1}^n p_i\right) = \sum\limits_{i=1}^{n}
\mu (p_i),
$$
for every family $p_1,\ldots, p_n$ of mutually orthogonal
projections in $M.$ If the set $\left\{ \|\mu (p)\|: p \in \mathcal{P} (M) \right\}$ is bounded, we shall say that $\mu$ is \emph{bounded}.\smallskip

The Bunce-Wright-Mackey-Gleason theorem \cite{BuWri92} affirms that if $M$ has no summand of type $I_2$, then every bounded finitely additive mapping $\mu: \mathcal{P} (M)\to X$ extends to a bounded linear operator from $M$ to $X$.\smallskip

Proposition 3.4 in \cite{NiPe} assures that, for each weak-2-local derivation $\Delta$ on a von Neumann algebra $M$, the measure $\mu_{_\Delta} : \mathcal{P} (M) \to M$ is finitely additive. The boundedness of this measure $\mu_{_\Delta}$ is, in general, an open problem.\smallskip

It is known that every family $(p_i)_{i \in I}$ of mutually orthogonal projections in
a von Neumann algebra $M$ is summable with respect to the weak$^*$ topology of  $M$
and $\displaystyle p = \hbox{weak}^*\hbox{-}\sum_{i\in I} p_i$ is a projection in $M$
(cf. \cite[Definition 1.13.4]{Sak}). We shall usually write $\displaystyle\sum_{i\in I} p_i$ instead of $\displaystyle\hbox{weak}^*\hbox{-}\sum_{i\in I} p_i$. Actually, the family $(p_i)$ satisfies a better summability condition, that is, the series $\displaystyle\sum_{i\in I} p_i$ also is summable with respect to the strong$^*$-topology of $M$, that is, $\displaystyle p= \hbox{weak}^*\hbox{-}\sum_{i\in I} p_i= \hbox{strong}^*\hbox{-}\sum_{i\in I} p_i$ (compare, for example, \cite[commments before Proposition 2.7]{KOPR2014}).\smallskip

We recall that, given a von Neumann algebra $M$, with predual $M_*$, the set of all normal states on $M$ (i.e. the set of all norm-one, positive functionals in $M_*$) is denoted by $S_{n} (M)$. Following standard notation \cite[Definition 1.8.6]{Sak}, the \emph{strong$^*$ topology} of $M$ (denoted by $s^*(M, M_*)$) is the locally convex topology on $M$ defined by the seminomrs $\||a|\|^2_{\phi} := \phi (\frac{ a a^* +a^* a}{2})$, $(a\in A),$ where $\phi$ runs in  $S_{n} (M)$.  The strong$^*$ topology of $M$ satisfies certain interesting properties, for example, a functional $\psi: M \to \mathbb{C}$ is strong$^*$ continuous if and only if it is weak$^*$ continuous (see \cite[Corollary 1.8.10]{Sak}). A consequence of the Grothendieck's inequality implies that a linear map between von Neumann algebras is strong$^*$ continuous if and only it is weak$^*$ continuous (cf. \cite[page 621]{PeRo}). Another interesting property of this topology asserts that the product of every von Neumann algebra is jointly strong$^*$ continuous on bounded sets (see \cite[Proposition 1.8.12]{Sak}). Finally, we also recall that given a von Neumann subalgebra $N$ of $M$, the strong$^*$-topology of $N$ coincides with the restriction to $N$ of the strong$^*$-topology of $M$, that is, $S^*(N,N_*) = S^* (M,M_*)|_{N}$ (cf. \cite[COROLLARY]{Bun01}).\smallskip

Let $M$ be a von Neumann algebra and let $\tau$ denote the weak$^*$ or the strong$^*$-topology. A function $\mu: \mathcal{P} (M)\to M$ is said to be \emph{$\tau$-completely additive} if
\begin{equation}\label{eq completely additive}
\mu\left(\hbox{weak$^*$-}\sum\limits_{i\in I} p_i\right) =
\hbox{$\tau$-}\sum\limits_{i\in I}\mu(p_i)
\end{equation} for every family $\{p_i\}_{i\in I}$ of mutually orthogonal projections in $\mathcal{P}(M),$ where the summability of the right hand side is with respect to the topology $\tau$. S. Dorofeev and A.N. Shertsnev supplemented the previous Bunce-Wright-Mackey-Gleason theorem by showing that every completely additive measure on the set of projections of a von Neumann algebra with no type $I_n$ ($n<\infty$) direct summands is bounded (compare \cite{Doro1990,Doro1990b,Doro,DoroShers1990} and \cite{Shers2008} or the monograph \cite{Dvurech}).\smallskip

\begin{theorem}\label{t weak-2-local derivations on B(l2)} Let $H$ denote a separable infinite dimensional complex Hilbert space, and let $\Delta: B(H)\to B(H)$ be a {\rm(}non-necessarily linear nor continuous{\rm)} weak-2-local derivation. Then $\Delta|_{B(H)_{sa}} : B(H)_{sa} \to B(H)$ is a linear map.
\end{theorem}

\begin{proof} Throughout this proof, $K(H)$ is regarded as a C$^*$-subalgebra and a closed two-sided ideal of $B(H)$. Lemma \ref{l ideals are fixed} implies that $\Delta (K(H)) \subseteq K(H)$ and $\Delta |_{K(H)} : K(H) \to K(H)$ is a weak-2-local derivation on $K(H)$.\smallskip

Let $p$ be a projection in $ B(H)$. We observe that $p$ might not belong to $K(H)$. Anyway, the separability of $H$ assures the existence of an at most countable (possibly finite) family $(p_n)_{n\in I}$ of minimal mutually orthogonal projections in $B(H)$ satisfying $\displaystyle p = \hbox{weak$^*$-}\sum_{n\in I} p_n$. Let $\mathcal{B}_{(p_n)}$ denote the commutative C$^*$-subalgebra of $K(H)$ generated by the $p_n$'s. Theorem \ref{t derivation on Ka} implies that $\Delta|_{\mathcal{B}_{(p_n)}} :\mathcal{B}_{(p_n)}\to K(H)$ is a continuous linear derivation. The bitransposed operator $(\Delta|_{\mathcal{B}_{(p_n)}})^{**} :\mathcal{B}_{(p_n)}^{**}\subseteq B(H)\to B(H)$ is a weak$^*$-continuous linear derivation (cf. \cite[Remark 2.6]{AyuKudPe2014}), and hence strong$^*$-continuous (cf. \cite[page 621]{PeRo}). We obviously know that $p\in \mathcal{B}_{(p_n)}^{**}$. It follows from these properties, and the fact that $S^*(B(H),B(H)_*)|_{\mathcal{B}_{(p_n)}^{**}} = S^* (\mathcal{B}_{(p_n)}^{**},\mathcal{B}_{(p_n)}^{*})$, that $$(\Delta|_{\mathcal{B}_{(p_n)}})^{**}(p) = \hbox{strong$^*$-}\sum_{n=1}^{\infty} \Delta (p_n).$$
We claim that $\Delta(p) =  (\Delta|_{\mathcal{B}_{(p_n)}})^{**}(p)$. Indeed, for each natural $N$, let $p_{N}$ denote $\displaystyle \sum_{k=1}^{N} p_k$. We already know that $\displaystyle \hbox{strong$^*$-}\lim_{N} \Delta (p_{N}) = (\Delta|_{\mathcal{B}_{(p_n)}})^{**}(p).$ By Lemma \ref{l pDelta(pap+b)p}, the identity $$(1-p+p_N) \Delta (p) (1-p+p_N) = (1-p+p_N) \Delta (p_N) (1-p+p_N),$$ holds for every natural $N$. Taking strong$^*$-limits in the above equality and having in mind the joint strong$^*$ continuity of the product of $B(H)$, we obtain \begin{equation}\label{eq Delta preserves strong* summs} \Delta(p) =  (\Delta|_{\mathcal{B}_{(p_n)}})^{**}(p)= \hbox{strong$^*$-}\sum_{n=1}^{\infty} \Delta (p_n).
\end{equation} \smallskip

Let us define a measure $$\mu : \mathcal{P} (B(H)) \to B(H)$$ $$p\mapsto \Delta(p).$$ Proposition 3.4 in \cite{NiPe} implies that $\mu$ is finitely additive. We claim that $\mu$ is a strong$^*$-completely additive vector measure on $\mathcal{P} (B(H))$. Indeed, let $(p_k)$ be an at most countable family of mutually orthogonal projections in $\mathcal{P} (B(H))$. By the separability of $H$, for each natural $k$, there exists an at most countable family $(p^k_n)_n$ of mutually orthogonal minimal projections in $B(H)$ such that $\displaystyle p_k =\sum_{n=1}^{\infty} p_n^k$. Let $\mathcal{B}_k$ denote the C$^*$-subalgebra of $K(H)$ generated by $\{p_n^k: n\in \mathbb{N}\}$, and let $\mathcal{B}$ be the C$^*$-subalgebra of $K(H)$ generated by $\{p_n^k: k,n\in \mathbb{N}\}.$ By \eqref{eq Delta preserves strong* summs} and the strong$^*$-continuity of $(\Delta|_{\mathcal{B}})^{**}$, we have $$\mu \left(\sum_{k=1}^{\infty} p_k\right) = \Delta \left(\sum_{k=1}^{\infty} p_k\right) = (\Delta|_{\mathcal{B}})^{**} \left(\sum_{k=1}^{\infty} p_k\right) = \hbox{strong$^*$-}\sum_{k=1}^{\infty}  (\Delta|_{\mathcal{B}})^{**} (p_k) =$$ $$= \hbox{strong$^*$-}\sum_{k=1}^{\infty}  \left(\hbox{strong$^*$-}\sum_{n=1}^{\infty} \Delta(p^k_n)\right) =  \hbox{strong$^*$-}\sum_{k=1}^{\infty}  (\Delta|_{\mathcal{B}_k})^{**} (p_k)$$ $$=  \hbox{strong$^*$-}\sum_{k=1}^{\infty}  \Delta (p_k)=  \hbox{strong$^*$-}\sum_{k=1}^{\infty}  \mu (p_k),$$ which proves that $\mu$ is a strong$^*$-completely additive measure.\smallskip

By the Mackey-Gleason theorem there exists a bounded linear operator $G: B(H) \to  B(H)$ satisfying that $G(p) = \mu(p) =\Delta(p)$, for every $p\in \mathcal{P} ( B(H))$. Theorem \ref{t derivation on Ka} combined with the spectral resolution of a compact self-adjoint operator in $B(H)$ imply that $$\Delta (a) = G(a),$$ for every $a\in K(H)_{sa}$.\smallskip

We claim that $G: B(H) \to  B(H)$ is weak$^*$-continuous. Let $(p_k)$ be a countable family of mutually orthogonal projections in $\mathcal{P} (B(H))$. We have shown above that $$G \left(\sum_{k=1}^{\infty} p_k\right) = \mu\left(\sum_{k=1}^{\infty} p_k\right) = \hbox{strong$^*$-}\sum_{k=1}^{\infty} \mu \left(p_k\right) = \hbox{strong$^*$-}\sum_{k=1}^{\infty} G\left(p_k\right).$$ We conclude from the separability of $H$ and Corollary III.3.11 in \cite{Takesaki} that $\varphi G\in M_*$, for every $\varphi\in M_*$, which implies that $G$ is weak$^*$-continuous and hence strong$^*$ continuous (cf. \cite[page 621]{PeRo}).\smallskip

Take now a self-adjoint operator $a\in B(H)$. Having in mind the separability of $H$, we can write $\displaystyle a =\sum_{n=1}^{\infty} \lambda_n p_n$, where $(\lambda_n)$ is a bounded sequence of real numbers, $(p_n)$ is a sequence of mutually orthogonal minimal projections in $B(H)$, and the series converges with respect to the strong$^*$-topology of $B(H)$.\smallskip

Arguing as above, let $\displaystyle p = \hbox{weak$^*$-}\sum_{n=1}^{\infty} p_n$, and let $\mathcal{B}_{(p_n)}$ denote the commutative C$^*$-subalgebra of $K(H)$ generated by the $p_n$'s. By Theorem \ref{t derivation on Ka}, the mapping $\Delta|_{\mathcal{B}_{(p_n)}} :\mathcal{B}_{(p_n)}\to K(H)$ is a continuous linear derivation. Therefore, the operator $(\Delta|_{\mathcal{B}_{(p_n)}})^{**} :\mathcal{B}_{(p_n)}^{**}\subseteq B(H)\to B(H)$ is a weak$^*$-continuous linear derivation (cf. \cite[Remark 2.6]{AyuKudPe2014}), and hence strong$^*$-continuous (cf. \cite[page 621]{PeRo}). In this case, $p,a\in \mathcal{B}_{(p_n)}^{**}$. We deduce from the strong$^*$-continuity of $(\Delta|_{\mathcal{B}_{(p_n)}})^{**}$, and the fact $S^*(B(H),B(H)_*)|_{\mathcal{B}_{(p_n)}^{**}} = S^* (\mathcal{B}_{(p_n)}^{**},\mathcal{B}_{(p_n)}^{*})$, that $$ (\Delta|_{\mathcal{B}_{(p_n)}})^{**}(a) = \hbox{strong$^*$-}\sum_{n=1}^{\infty} \lambda_n \Delta (p_n).$$ We shall prove that $\Delta(a) =  (\Delta|_{\mathcal{B}_{(p_n)}})^{**}(a)$, and hence \begin{equation}\label{eq Delta and G coincide on selfadjoints} \Delta(a) =  (\Delta|_{\mathcal{B}_{(p_n)}})^{**}(a) = \hbox{strong$^*$-}\sum_{n=1}^{\infty} \lambda_n \Delta (p_n) = \hbox{strong$^*$-}\sum_{n=1}^{\infty} \lambda_n G (p_n)
 \end{equation}$$= \hbox{strong$^*$-}\lim_{N\to\infty} G \left( \sum_{n=1}^{N} \lambda_n p_n\right)= \hbox{(by the strong$^*$-continuity of $G$)} = G(a).$$ To this end, for each natural $N$, let $p_{N}$ denote $\displaystyle \sum_{k=1}^{N} p_k$. We already know that $\displaystyle \hbox{strong$^*$-}\lim_{N} \Delta \left( \sum_{n=1}^{N} \lambda_n p_n\right) = (\Delta|_{\mathcal{B}_{(p_n)}})^{**}(a).$ Lemma \ref{l pDelta(pap+b)p} implies that $$(1-p+p_N) \Delta (a) (1-p+p_N) = (1-p+p_N) \Delta \left( \sum_{n=1}^{N} \lambda_n p_n\right) (1-p+p_N),$$ for every natural $N$. Taking strong$^*$-limits in the above identity, it follows from the joint strong$^*$ continuity of the product in $B(H)$ that $$ \Delta(a) =  (\Delta|_{\mathcal{B}_{(p_n)}})^{**}(a),$$ as desired. \smallskip

We have proved that $\Delta (a) = G(a)$, for every $a\in B(H)_{sa}$ (compare \eqref{eq Delta and G coincide on selfadjoints}). Therefore, $\Delta (a+b) = G(a+b) = G(a) + G(b) = \Delta (a) +\Delta (b)$, for every $a,b\in B(H)_{sa}$.\end{proof}

Our last result is a direct consequence of the previous Theorem \ref{t weak-2-local derivations on B(l2)} and \cite[Lemma 2.3]{NiPe}.

\begin{theorem}\label{t weak-2-local *derivations on B(l2)} Let $H$ be a separable complex Hilbert space. Then
every {\rm(}non-necessarily linear nor continuous{\rm)} weak-2-local $^*$-derivation on $B(H)$ is linear and a $^*$-derivation.$\hfill\Box$
\end{theorem}

\end{document}